\documentclass[oneside,reqno,english]{amsart}
\usepackage[T1]{fontenc}
\usepackage[utf8]{inputenc}
\usepackage{xcolor}
\usepackage{babel}
\usepackage{prettyref}
\usepackage{amstext}
\usepackage{amsthm}
\usepackage{amssymb}
\usepackage[pdfusetitle,
 bookmarks=true,bookmarksnumbered=false,bookmarksopen=false,
 breaklinks=false,pdfborder={0 0 0},pdfborderstyle={},backref=false,colorlinks=false]
 {hyperref}
\hypersetup{
 colorlinks=true,citecolor=blue,linkcolor=blue,linktocpage=true}

\makeatletter
\numberwithin{equation}{section}
\numberwithin{figure}{section}

\usepackage{prettyref}

\newrefformat{cor}{Corollary~\ref{#1}}
\newrefformat{subsec}{Section~\ref{#1}}
\newrefformat{lem}{Lemma~\ref{#1}}
\newrefformat{thm}{Theorem~\ref{#1}}
\newrefformat{sec}{Section~\ref{#1}}
\newrefformat{chap}{Chapter~\ref{#1}}
\newrefformat{prop}{Proposition~\ref{#1}}
\newrefformat{exa}{Example~\ref{#1}}
\newrefformat{tab}{Table~\ref{#1}}
\newrefformat{rem}{Remark~\ref{#1}}
\newrefformat{def}{Definition~\ref{#1}}
\newrefformat{fig}{Figure~\ref{#1}}
\newrefformat{claim}{Claim~\ref{#1}}

\makeatother

\theoremstyle{plain}
\newtheorem{thm}{\protect\theoremname}[section]
\theoremstyle{remark}
\newtheorem{rem}[thm]{\protect\remarkname}
\theoremstyle{plain}
\newtheorem{prop}[thm]{\protect\propositionname}
\newtheorem{cor}[thm]{\protect\corollaryname}
\newtheorem{assumption}[thm]{\protect\assumptionname}
\providecommand{\assumptionname}{Assumption}
\providecommand{\corollaryname}{Corollary}
\providecommand{\propositionname}{Proposition}
\providecommand{\remarkname}{Remark}
\providecommand{\theoremname}{Theorem}

\begin{document}
\subjclass[2020]{Primary: 42C15; secondary: 47A63, 47B65}
\title[Random Frames from WR Flows]{Random Frame Decompositions from Weighted Residual Flows}
\begin{abstract}
We study the evolution of a positive operator under weighted residual
maps determined by a finite family of orthogonal projections. Iterating
these maps along the rooted tree of multi-indices produces a “weighted
residual energy tree”, together with natural path measures obtained
by normalizing the dissipated energy or trace at each step. Under
a quantitative coverage condition on the projections, we show that
along almost every branch the residuals converge strongly to zero
and the dissipated pieces admit a rank-one decomposition that reconstructs
the initial operator. In the special case where the initial operator
is the identity on a subspace, this yields almost surely a random
Parseval frame generated intrinsically by the weighted residual dynamics.
\end{abstract}

\author{James Tian}
\address{Mathematical Reviews, 535 W. William St, Suite 210, Ann Arbor, MI
48103, USA}
\email{james.ftian@gmail.com}
\keywords{weighted residual maps; positive operators; random frames; path measures;
branching processes}

\maketitle
\tableofcontents{}

\section{Introduction}\label{sec:1}

Let $H$ be a Hilbert space and let $\left\{ P_{1},\dots,P_{m}\right\} $
be a fixed finite family of orthogonal projections on $H$. Given
a positive operator $R\in B\left(H\right)_{+}$, we consider the basic
update 
\[
\Phi_{P}\left(R\right):=R^{1/2}\left(I-P\right)R^{1/2},
\]
which removes the $P$-component of $R$ and transports the remainder
back into the cone $B\left(H\right)_{+}$. Iterating these maps along
the rooted $m$-ary tree of multi-indices produces a tree-indexed
family of residual operators and a corresponding family of dissipated
pieces. The resulting object, which we call the \emph{weighted residual
(WR) energy tree}, packages a collection of Loewner-monotone identities
that express $R$ as a telescoping sum of positive contributions over
finite levels of the tree. The general background on positive operators
and Loewner order used throughout is standard \cite{MR2284176,MR493419}.

The point of view of this paper is that the WR energy tree is not
merely an algorithmic bookkeeping device, but also a canonical sample
space for an intrinsic stochastic process. At each node $w$ of the
tree, the update from $R_{w}$ to its children $R_{wj}$ is accompanied
by a dissipated operator 
\[
D_{w,j}:=R^{1/2}_{w}P_{j}R^{1/2}_{w}\ge0,
\]
and these dissipated quantities may be normalized to produce transition
probabilities on the set of children. In this way the WR dynamics
canonically generate probability measures on the boundary $\Omega=\left\{ 1,\dots,m\right\} ^{\mathbb{N}}$,
with no extrinsic randomness imposed. We work with two complementary
normalizations, and the probabilistic tools we use are standard martingale
and branching process arguments \cite{MR1155402,MR3930614,MR163361,MR3616205}.
\begin{enumerate}
\item \emph{Energy-biased branching.} Fix $x\in H$ with $\left\langle x,R_{0}x\right\rangle >0$.
By normalizing the scalar weights $\left\langle x,D_{w,j}x\right\rangle $
at each step, we obtain a consistent family of cylinder probabilities
and hence a path measure $\nu_{x}$ on $\Omega$. This measure reflects
the manner in which the quadratic form $\left\langle x,Rx\right\rangle $
is dissipated down the tree.
\item \emph{Trace-biased branching.} When $R_{0}$ is trace class, we may
instead normalize the weights $\mathrm{tr}\left(D_{w,j}\right)$ to
obtain a measure $\ensuremath{\nu_{\mathrm{tr}}}$ on $\Omega$. This
is the natural choice for reconstruction statements at the level of
operators, rather than a fixed vector state.
\end{enumerate}
\noindent The first main theorem is an extinction principle under
a quantitative nondegeneracy hypothesis on the projections. Informally,
we assume that the family $\left\{ P_{j}\right\} $ captures a definite
fraction of the energy of every residual vector that appears along
the WR orbit. In the paper this is stated as a uniform leakage condition
\eqref{eq:5-1}, and it implies that along $\nu_{x}$-almost every
branch the residual energy tends to zero: 
\[
\left\langle x,R_{\infty}\left(\omega\right)x\right\rangle =0\qquad\text{for \ensuremath{\nu_{x}}-almost every \ensuremath{\omega\in\Omega}.}
\]
In particular, the pathwise telescoping identity becomes an exact
energy decomposition for $\left\langle x,R_{0}x\right\rangle $ along
typical branches. We prove this using elementary martingale and branching
process arguments adapted to the intrinsic weights of the WR tree.

The second main theorem is an operator level reconstruction from rank-one
atoms. Assuming $R_{0}$ is trace class, each dissipated operator
along a branch is compact and admits a spectral decomposition into
rank-one pieces. By arranging these pieces into a countable collection
of vectors $\left\{ \varphi_{k,r}\left(\omega\right)\right\} $ on
the branch, we obtain a pathwise Parseval-type identity 
\[
\sum_{k,r}\left|\left\langle x,\varphi_{k,r}\left(\omega\right)\right\rangle \right|^{2}=\left\langle x,R_{0}x\right\rangle ,
\]
valid for every $x\in H$ and for $\ensuremath{\nu_{\mathrm{tr}}}$-almost
every $\omega$. Equivalently, the corresponding random frame operator
converges in the strong operator topology to $R_{0}$. In the special
case that $R_{0}$ is the identity on a closed subspace $K\subseteq H$
(and vanishes on $K^{\perp}$), the vectors $\left\{ \varphi_{k,r}\left(\omega\right)\right\} $
form a Parseval frame for $K$ for $\nu_{\mathrm{tr}}$-almost every
branch. Thus, starting from a single positive operator and a finite
family of projections, the WR dynamics generate an intrinsically random
Parseval frame whose frame operator is exactly the initial datum.
This places the work in the orbit of classical frame theory and its
operator-theoretic formulations \cite{MR1686653,MR1968126,MR2066823,MR2224902,MR4832149,MR3005286,MR4953099,MR4280112}.

\subsection*{Relation to existing literature}

From the perspective of frame theory, this work lies at the intersection
of (i) classical frames and their operator-theoretic formulations
\cite{MR1686653,MR1968126,MR2224902}, (ii) fusion frames and frames
of subspaces motivated by distributed reconstruction \cite{MR2066823},
and (iii) dynamical and algorithmic viewpoints in which the frame
coefficients are generated from iterates of an underlying operator
\cite{MR3027915,MR3613395,MR4093918,MR3785475,MR3650440,MR3582259,MR4990268}.
A common feature of these directions is that the frame object is specified
\emph{a priori} (as a family of vectors, subspaces, or a sampling
distribution), and one studies stability, optimality, or reconstruction
properties. The present paper differs in that the frame is \emph{generated
dynamically} from the Loewner-monotone dissipation of a single operator
under repeated residual updates.

On the operator side, $\Phi_{P}\left(R\right)=R^{1/2}\left(I-P\right)R^{1/2}$
belongs to the same formal neighborhood as shorted operators, Schur
complements, and related residual constructions in matrix analysis
and network theory \cite{MR287970,MR356949,MR938493,MR852902,MR2234254,MR2345997,MR3983070,MR242573}.
Those theories are typically static (one shorts once relative to a
single subspace, or takes a single Schur complement of a block operator),
whereas the WR energy tree iterates a residual step along a fixed
flag and propagates the result down an infinite branching structure.
What's new here is the combination of Loewner-order telescoping with
intrinsic boundary measures, producing almost-sure extinction and,
in the trace-class setting, almost-sure Parseval reconstruction. Classical
geometric and analytic structure of pairs of subspaces and projections
provides a useful conceptual reference point \cite{MR34514,MR251519,MR2580440},
as does the general theory of operator ranges and range inclusions
\cite{MR203464,MR293441,MR531986,MR866966}. (We emphasize that we
do not require any commuting or two-projection normal form assumptions.)

Finally, the appearance of probability measures supported on the boundary
of a rooted tree places the construction close in spirit to the general
theory of stochastic processes on trees. The measures $\nu_{x}$ and
$\nu_{\mathrm{tr}}$ arise canonically from the WR weights, and the
proofs use Kolmogorov consistency together with martingale techniques
adapted to this operator-generated branching.

The WR update is also naturally adjacent to projection algorithms
and randomized iterative schemes, particularly those connected to
Kaczmarz-type methods and cyclic/random products of projections \cite{MR2311862,MR2500924,MR3134343,MR3439812,MR4235310,MR3432148,MR3773065,MR2903120,MR3145756,MR4310540,MR4997577}.
We do not consider algorithmic complexity questions here, but the
structural parallel helps motivate the role of normalized dissipation
as a canonical sampling rule.

\subsection*{Organization}

\prettyref{sec:2} sets up the WR energy tree and records the basic
telescoping identities over finite levels. \prettyref{sec:3} constructs
the boundary path measures from energy normalized dissipation and
proves Kolmogorov consistency, including the treatment of finite-time
absorption. \prettyref{sec:4} develops the random branch bookkeeping
identities that relate residuals to cumulative dissipation. \prettyref{sec:5}
proves uniform extinction under the leakage hypothesis. \prettyref{sec:6}
passes from operator-valued dissipation to rank-one atoms along a
branch. \prettyref{sec:7} uses trace-biased branching to obtain the
almost-sure Parseval frame and the strong operator reconstruction
of $R_{0}$.

\section{WR energy trees}\label{sec:2}

Let $H$ be a complex Hilbert space. For any projection $P$ and any
positive operator $R\in B(H)_{+}$, define the weighted residual map
\begin{equation}
\Phi_{P}(R):=R^{1/2}\left(I-P\right)R^{1/2}.\label{eq:b1}
\end{equation}
Given an initial positive operator $R_{0}\in B(H)_{+}$ and a sequence
of projections $\left(P_{n}\right)_{n\ge1}$, the WR iteration is
defined by 
\begin{equation}
R_{n+1}=\Phi_{P_{n+1}}\left(R_{n}\right),\quad n\ge0,
\end{equation}
with $R_{0}$ fixed. For such a sequence, it is known that $0\le R_{n+1}\le R_{n}\le R_{0}$
for all $n$, and that $\left(R_{n}\right)$ converges strongly to
a limit $R_{\infty}$. Moreover, the one-step decomposition 
\begin{equation}
R_{n}-R_{n+1}=R^{1/2}_{n}P_{n+1}R^{1/2}_{n}\ge0
\end{equation}
yields the global energy identity via a telescoping argument: 
\begin{equation}
R_{0}=R_{\infty}+\sum^{\infty}_{n=0}D_{n},\quad D_{n}:=R^{1/2}_{n}P_{n+1}R^{1/2}_{n},\label{eq:b-4}
\end{equation}
where the series converges in the strong operator topology (see e.g.,
\cite{tian2025alternating,tian2025WR}). 

We now lift this construction from a single sequence of projections
to a rooted tree indexed by all possible finite sequences of indices
from $\left\{ 1,\dots,m\right\} $.

\subsection*{WR energy tree}

Fix an integer $m\geq2$ and orthogonal projections 
\[
\left\{ P_{1},\dots,P_{m}\right\} 
\]
on $H$. Let the alphabet be $\mathcal{A}:=\left\{ 1,\dots,m\right\} $.
A finite word is a sequence $w=j_{1}j_{2}\dots j_{n}$ with $j_{k}\in\mathcal{A}$
and $n\ge1$. The empty word is denoted by $\emptyset$. The set of
all finite words is 
\[
\mathcal{W}:=\left\{ \emptyset\right\} \cup\bigcup_{n\ge1}\mathcal{A}^{n}.
\]
For a nonempty word $w=j_{1}\dots j_{n}$, its prefix of length $n-1$
is denoted $w^{-}:=j_{1}\dots j_{n-1}$. For any word $w$ and letter
$j\in\mathcal{A}$, we write $wj$ for the concatenation. An infinite
word, or branch, is a sequence $\omega=\left(j_{1},j_{2},\dots\right)\in\mathcal{A}^{\mathbb{N}}$.
Its prefix of length $n$ is written as $\omega|_{n}:=j_{1}\dots j_{n}\in\mathcal{A}^{n}$.

We define a family of operators indexed by $\mathcal{W}$. At the
root, set $R_{\emptyset}:=R_{0}$. For a word $w=j_{1}\dots j_{n}$
of length $n\ge1$, we define the residual operator $R_{w}$ recursively
by 
\[
R_{w}:=\Phi_{P_{j_{n}}}\left(R_{w^{-}}\right)=R^{1/2}_{w^{-}}\left(I-P_{j_{n}}\right)R^{1/2}_{w^{-}}.
\]
Additionally, for a nonempty word $w=j_{1}\dots j_{n}$, we define
the associated dissipated piece 
\[
D_{w}:=R^{1/2}_{w^{-}}P_{j_{n}}R^{1/2}_{w^{-}}.
\]
By construction, we have the identity $R_{w^{-}}-R_{w}=D_{w}$ for
every nonempty word $w$. Each $R_{w}$ and $D_{w}$ is positive,
and they satisfy $0\le D_{w}\le R_{w^{-}}\le R_{0}$ in the Loewner
order. We refer to the family $\left\{ R_{w},D_{w}:w\in\mathcal{W}\right\} $
as the WR energy tree.
\begin{thm}[pathwise energy decomposition]
\label{thm:B-1} Let $H$ be a Hilbert space, let $P_{1},\dots,P_{m}$
be orthogonal projections on $H$, and let $R_{0}\in B(H)_{+}$. Construct
the family $\left\{ R_{w},D_{w}\right\} $ as above. Then the following
hold.
\begin{enumerate}
\item For every word $w\in\mathcal{W}$ and every letter $j\in\mathcal{A}$,
we have $R_{wj}\le R_{w}$ and $D_{wj}\le R_{w}$. Along any chain
of extensions $w_{0}\subset w_{1}\subset\dots\subset w_{n}$, the
sequence 
\[
R_{w_{0}}\ge R_{w_{1}}\ge\dots\ge R_{w_{n}}\ge0
\]
is decreasing in the Loewner order. In particular, for any infinite
word $\omega=\left(j_{1},j_{2},\dots\right)$, the path operators
$R_{\omega|_{n}}$ form a decreasing sequence of positive operators
bounded by $R_{0}$.
\item For each infinite word $\omega\in\mathcal{A}^{\mathbb{N}}$, there
exists a unique positive operator $R_{\infty}\left(\omega\right)\in B(H)_{+}$
such that 
\[
R_{\omega|_{n}}\xrightarrow[n\to\infty]{s}R_{\infty}\left(\omega\right).
\]
Therefore, for any $x\in H$, $\Vert R_{\omega|_{n}}x-R_{\infty}\left(\omega\right)x\Vert\to0$,
and the scalar sequence $\left\langle x,R_{\omega|_{n}}x\right\rangle $
is decreasing and converges to $\left\langle x,R_{\infty}\left(\omega\right)x\right\rangle $.
\item For any infinite word $\omega$, we have the infinite-path identity
\[
R_{0}=R_{\infty}\left(\omega\right)+\sum^{\infty}_{k=1}D_{\omega|_{k}},
\]
where the series converges in the strong operator topology. The partial
sums $\sum^{n}_{k=1}D_{\omega|_{k}}$ form an increasing sequence
of positive operators bounded by $R_{0}$.
\item For any infinite word $\omega$ and any vector $x\in H$, the scalar
quantities $d_{k}\left(\omega,x\right):=\left\langle x,D_{\omega|_{k}}x\right\rangle \ge0$
satisfy 
\[
\sum^{\infty}_{k=1}d_{k}\left(\omega,x\right)=\left\langle x,R_{0}x\right\rangle -\left\langle x,R_{\infty}\left(\omega\right)x\right\rangle .
\]
\end{enumerate}
\end{thm}

\begin{proof}
We proceed in steps mirroring the basic WR iteration arguments in
\eqref{eq:b1}--\eqref{eq:b-4}.

Fix a word $w\in\mathcal{W}$ and a letter $j\in\mathcal{A}$. By
definition, $R_{wj}=R^{1/2}_{w}\left(I-P_{j}\right)R^{1/2}_{w}$.
For any $x\in H$, 
\[
\left\langle x,R_{wj}x\right\rangle =\Vert\left(I-P_{j}\right)R^{1/2}_{w}x\Vert^{2}\ge0,
\]
so $R_{wj}\ge0$. Furthermore, 
\[
\left\langle x,\left(R_{w}-R_{wj}\right)x\right\rangle =\langle x,R^{1/2}_{w}P_{j}R^{1/2}_{w}x\rangle=\Vert P_{j}R^{1/2}_{w}x\Vert^{2}\ge0.
\]
Hence $R_{w}-R_{wj}\ge0$, establishing $R_{wj}\le R_{w}$. The definition
$D_{wj}=R_{w}-R_{wj}$ implies $0\le D_{wj}\le R_{w}$. By induction,
along any chain $w_{0}\subset w_{1}\subset\dots\subset w_{n}$, the
sequence of operators is decreasing.

Fix $\omega=\left(j_{1},j_{2},\dots\right)$. Set $R_{n}:=R_{\omega|_{n}}$.
It follows that $0\le R_{n+1}\le R_{n}\le R_{0}$ for all $n\ge0$.
Since $\left(R_{n}\right)$ is a bounded monotone sequence of self-adjoint
operators, it converges strongly to a unique positive operator $R_{\infty}\left(\omega\right)$;
see e.g., \cite{MR493419}.

Let $D_{n}:=D_{\omega|_{n}}$. By definition, $D_{n}=R_{n-1}-R_{n}$.
Summing from $n=1$ to $N$ yields 
\[
\sum^{N}_{n=1}D_{n}=R_{0}-R_{N}.
\]
The partial sums of the series $\sum D_{n}$ are increasing and bounded
above by $R_{0}$, so they converge strongly to an operator $D\left(\omega\right)$.
Taking limits as $N\to\infty$, we obtain $D\left(\omega\right)=R_{0}-R_{\infty}\left(\omega\right)$,
which proves (3).

For arbitrary $x\in H$, taking inner products in the finite telescoping
identity gives 
\[
\left\langle x,R_{0}x\right\rangle =\left\langle x,R_{\omega|_{N}}x\right\rangle +\sum^{N}_{k=1}\left\langle x,D_{\omega|_{k}}x\right\rangle .
\]
Letting $N\to\infty$, and using the strong convergence established
in Step 2, we obtain 
\[
\left\langle x,R_{0}x\right\rangle =\left\langle x,R_{\infty}\left(\omega\right)x\right\rangle +\sum^{\infty}_{k=1}d_{k}\left(\omega,x\right),
\]
which is equivalent to the statement in (4). 
\end{proof}

\subsection*{Boundary space and cylinder sets}

Let 
\[
\Omega:=\mathcal{A}^{\mathbb{N}}=\left\{ \left(j_{1},j_{2},\dots\right):j_{n}\in\mathcal{A}\right\} 
\]
be the space of all infinite words, equipped with the product topology.
For each finite word $w=j_{1}\dots j_{n}\in\mathcal{W}$ we define
the cylinder set 
\[
\left[w\right]:=\left\{ \omega\in\Omega:\omega|_{n}=w\right\} .
\]
The collection of cylinders $\left\{ \left[w\right]:w\in\mathcal{W}\right\} $
forms a base for the product topology on $\Omega$. The associated
Borel $\sigma$-algebra will be denoted by $\mathcal{B}\left(\Omega\right)$.

For each $n\ge1$, the coordinate projection 
\[
\pi_{n}:\Omega\to\mathcal{A},\quad\pi_{n}\left(\omega\right)=j_{n}
\]
is continuous; the $\sigma$-algebra generated by $\left\{ \pi_{1},\dots,\pi_{n}\right\} $
coincides with the $\sigma$-algebra generated by the cylinders of
length $n$.

\subsection*{Boundary WR maps}

For each $\omega\in\Omega$ and $n\ge0$, we write 
\[
R_{n}\left(\omega\right):=R_{\omega|_{n}},\qquad D_{n}\left(\omega\right):=\begin{cases}
D_{\omega|_{n}}, & n\ge1,\\
0, & n=0,
\end{cases}
\]
with the convention $R_{\omega|_{0}}:=R_{\emptyset}=R_{0}$. The pathwise
limit from \prettyref{thm:B-1} gives a map 
\[
R_{\infty}:\Omega\to B(H)_{+},\quad\omega\mapsto R_{\infty}\left(\omega\right),
\]
characterized by the strong convergence $R_{n}\left(\omega\right)\xrightarrow{\text{s}}R_{\infty}\left(\omega\right)$
as $n\to\infty$ for each fixed $\omega$.

\section{Energy-biased branching and path measures}\label{sec:3}

We continue with the setting and notation from the previous section.
For each word $w\in\mathcal{W}$ we have the WR operators 
\[
R_{w}\in B(H)_{+},\qquad D_{wj}:=R^{1/2}_{w}P_{j}R^{1/2}_{w}\in B(H)_{+}
\]
for $j\in\mathcal{A}=\left\{ 1,\dots,m\right\} $. 

Fix a nonzero vector $x\in H$ with $\left\langle x,R_{0}x\right\rangle >0$.
Define the energy coefficients 
\[
a_{wj}:=\left\langle x,D_{wj}x\right\rangle \ge0.
\]
These numbers represent the amount of $x$-energy dissipated through
channel $j$ at node $w$. 

We call a node $w$ alive if 
\[
\sum^{m}_{j=1}a_{wj}>0.
\]
For any alive node $w$, define the energy-biased transition probabilities
\[
p_{x}(j\mid w):=\frac{a_{wj}}{\sum^{m}_{k=1}a_{wk}},\qquad j\in\mathcal{A}.
\]
By construction, 
\[
p_{x}(j\mid w)\ge0,\qquad\sum^{m}_{j=1}p_{x}(j\mid w)=1.
\]

If $\sum_{j}a_{wj}=0$, then necessarily $a_{wj}=0$ for every $j$,
so no further $x$-energy is dissipated from $R_{w}$ into any child.
In this case we say that $w$ is $x$-dead, and we fix, once and for
all, an arbitrary reference probability vector $q=\left(q_{1},\dots,q_{m}\right)$,
and define 
\[
p_{x}(j\mid w):=q_{j},\qquad j\in\mathcal{A}.
\]

The particular choice of $q$ will play no role in the results below:
all statements about energy dissipation and residual limits depend
only on the biased transitions at $x$-alive nodes, not on the arbitrary
continuation rule at $x$-dead nodes. Thus $p_{x}(\cdot\mid w)$ is
a bona fide probability distribution on $\mathcal{A}$ for every node
$w\in\mathcal{W}$. At alive nodes it is determined intrinsically
by the WR energy dissipation, while at $x$-dead nodes it is arbitrary
but normalized.

For a finite word $w=j_{1}\cdots j_{n}$, define the cylinder weight
\[
\nu_{x}([w]):=p_{x}(j_{1}\mid\emptyset)p_{x}(j_{2}\mid j_{1})\cdots p_{x}(j_{n}\mid j_{1}\cdots j_{n-1}).
\]
This defines a finitely additive set function on the semi-ring of
cylinder sets. For $\omega\in\Omega$ and $n\ge0$, we use the shorthand
\[
D_{\omega|n,j}:=D_{(\omega|n)j},\qquad j\in\mathcal{A},
\]
for the dissipated pieces along the $j$-labelled edges leaving the
node $\omega|n$.
\begin{thm}[WR path measures]
\textup{} \label{thm:c-4} There exists a unique Borel probability
measure $\nu_{x}$ on $\Omega=\mathcal{A}^{\mathbb{N}}$ whose values
on cylinder sets are given by the formula above. Moreover, for each
$n$, 
\[
\nu_{x}\left(\left\{ \omega:\sum\nolimits^{m}_{j=1}\left\langle x,D_{\omega|n,j}x\right\rangle =0\right\} \right)=\sum_{\substack{|w|=n\\
\sum_{j}a_{wj}=0
}
}\nu_{x}([w]).
\]
\end{thm}

\begin{proof}
Since $p_{x}(\cdot\mid w)$ is a probability distribution on $\mathcal{A}$
for every $w\in\mathcal{W}$, the cylinder weights satisfy the consistency
relation 
\[
\nu_{x}([w])=\sum^{m}_{j=1}\nu_{x}([wj]),\qquad w\in\mathcal{W}.
\]
Indeed, 
\[
\sum^{m}_{j=1}\nu_{x}([wj])=\nu_{x}([w])\sum^{m}_{j=1}p_{x}(j\mid w)=\nu_{x}([w]).
\]
By the Kolmogorov-Carathéodory extension theorem, there exists a unique
Borel probability measure $\nu_{x}$ on $\Omega$ extending these
cylinder weights. The second statement follows because the event 
\[
\left\{ \omega:\sum\nolimits_{j}\left\langle x,D_{\omega|n,j}x\right\rangle =0\right\} 
\]
is exactly the disjoint union of cylinders $[w]$ over all words $w$
of length $n$ that are $x$-dead. 
\end{proof}
\begin{rem}
The measure $\nu_{x}$ describes a random WR branch in which, at an
$x$-alive node $w$, the next projection index $j$ is selected with
probability proportional to the $x$-energy $\left\langle x,D_{wj}x\right\rangle $
dissipated by $P_{j}$ at that node. At $x$-dead nodes the branching
rule is no longer dictated by energy, and the choice of the next index
is arbitrary but normalized. Thus the randomness is entirely operator-theoretic
up to the point where the $x$-energy dynamics terminate. 
\end{rem}

\begin{rem}
Fix $w\in\mathcal{W}$ such that $\sum^{m}_{j=1}a_{wj}=0$. In operator
terms this means 
\[
\left\langle x,R^{1/2}_{w}P_{j}R^{1/2}_{w}x\right\rangle =0\quad\text{for all }j,
\]
or equivalently, 
\[
P_{j}R^{1/2}_{w}x=0\quad\text{for all }j.
\]
Thus the residual vector $R^{1/2}_{w}x$ lies entirely in the common
kernel 
\[
\bigcap^{m}_{j=1}\ker P_{j},
\]
and none of the channels $P_{1},\dots,P_{m}$ remove any further $x$-energy
at node $w$. From $\sum_{j}a_{wj}=0$ one cannot conclude that $R_{w}x=0$,
nor that $D_{wj}$ vanish as operators. The statement concerns only
the action of the WR tree on the specific state $x$; other vectors
$y\in H$ may still experience nonzero dissipation through some $P_{j}$.
Accordingly, the operator-valued WR tree $\{R_{wj},D_{wj}\}$ may
continue to evolve below $w$, even though the $x$-energy dynamics
have terminated there. 
\end{rem}

\subsection{Special cases: splitting and binary complementarity}\label{subsec:c-1}

We now discuss situations in which the WR branching probabilities
simplify or admit closed-form expressions.

\subsubsection*{Splitting on the evolving support}

Since the WR iterates satisfy $0\leq R_{w}\leq R_{0}$ for all words
$w$, the Douglas range-inclusion theorem (\cite{MR203464}) implies
that 
\begin{equation}
ran\,(R^{1/2}_{w})\subset ran\,(R^{1/2}_{0}),\quad w\in\mathcal{W}.\label{eq:c-1}
\end{equation}
Thus all evolving supports of the WR tree lie inside the fixed initial
support 
\begin{equation}
H_{0}:=\overline{ran}(R^{1/2}_{0}).\label{eq:c-2}
\end{equation}

Suppose now that the projections $P_{1},\dots,P_{m}$ satisfy the
splitting condition: 
\begin{equation}
\sum^{m}_{j=1}P_{j}=I\quad\text{on }H_{0}.\label{eq:c-3}
\end{equation}
It follows that $\sum^{m}_{j=1}P_{j}=I$ on $\overline{ran}(R^{1/2}_{w})$
for every $w\in\mathcal{W}$. Then we obtain the decomposition
\[
R_{w}=\sum^{m}_{j=1}D_{wj},\quad\left\langle x,R_{w}x\right\rangle =\sum^{m}_{j=1}a_{wj}
\]
In particular, for alive nodes, 
\[
p_{x}\left(j\mid w\right)=\frac{a_{wj}}{\sum^{m}_{k=1}a_{wk}}=\frac{\left\langle x,D_{wj}x\right\rangle }{\left\langle x,R_{w}x\right\rangle }.
\]
Thus under splitting hypothesis, the probability of choosing channel
$j$ at node $w$ is exactly the fraction of the current residual
energy of $x$ that is dissipated via $P_{j}$.

\subsubsection*{The binary case}

Assume $m=2$, and $P_{1}+P_{2}=I$ on $H_{0}$. Then $D_{w1}+D_{w2}=R_{w}$,
and we have the cross-identities 
\[
R_{w1}=D_{w2},\qquad R_{w2}=D_{w1}.
\]
Consequently, the residual energies are additive: 
\[
\left\langle x,R_{w}x\right\rangle =\left\langle x,R_{w1}x\right\rangle +\left\langle x,R_{w2}x\right\rangle .
\]
This conservation of residual energy allows us to define a \textit{residual
path measure} $\mu_{x}$, distinct from the dissipation measure $\nu_{x}$.
Define 
\[
\mu_{x}\left(\left[w\right]\right):=\frac{\left\langle x,R_{w}x\right\rangle }{\left\langle x,R_{0}x\right\rangle }.
\]
Then 
\[
\mu_{x}\left(\left[w\right]\right)=\mu_{x}\left(\left[w1\right]\right)+\mu_{x}\left(\left[w2\right]\right).
\]
Since $\mu_{x}\left(\left[\emptyset\right]\right)=1$ (assuming $\left\Vert x\right\Vert =1$),
so $\mu_{x}$ extends to a cylinder-additive probability measure on
$\Omega$.
\begin{rem}[Comparison of measures]
\textit{} In this binary setting, we have two natural measures:
\begin{enumerate}
\item The \textit{dissipation measure} $\nu_{x}$ (from \prettyref{thm:c-4}),
driven by $D_{wj}$, which tracks the path of captured energy.
\item The \textit{residual measure} $\mu_{x}$, driven by $R_{wj}$, which
tracks the path of surviving energy.
\end{enumerate}
Because $D_{w1}=R_{w2}$, these measures are conjugate. That is, the
probability of dissipating into channel $1$ is proportional to the
residual weight of channel $2$: 
\[
p_{x}\left(1\mid w\right)=\frac{\left\langle x,D_{w1}x\right\rangle }{\left\langle x,R_{w}x\right\rangle }=\frac{\left\langle x,R_{w2}x\right\rangle }{\left\langle x,R_{w}x\right\rangle }=\frac{\mu_{x}\left(\left[w2\right]\right)}{\mu_{x}\left(\left[w\right]\right)}.
\]
The residual measure $\mu_{x}$ is special here because it admits
the explicit closed-form representation $\mu_{x}\left(\left[w\right]\right)=\left\langle x,R_{w}x\right\rangle /\left\langle x,R_{0}x\right\rangle $,
which is not available in the general $m$ case.
\end{rem}

\section{Random branches and energy balance}\label{sec:4}

We continue to work with a fixed nonzero state $x\in H$ with $\left\langle x,R_{0}x\right\rangle >0$,
and with the WR tree $\{R_{w},D_{wj}\}$ and the path measure $\nu_{x}$
constructed in \prettyref{sec:3}. Throughout this section we equip
$\Omega=\mathcal{A}^{\mathbb{N}}$ with the canonical filtration $\{\mathcal{F}_{n}\}_{n\ge0}$,
where $\mathcal{F}_{n}$ is the $\sigma$-algebra generated by cylinder
sets of length $n$.

\subsection{Measurability and random boundary operators}

Our first observation is that the WR boundary limit $R_{\infty}(\omega)$
can be viewed as a random positive operator in a natural sense.
\begin{prop}
\label{prop:d-1} For all $x,y\in H$, the functions 
\[
\omega\longmapsto\left\langle x,R_{\omega|n}y\right\rangle ,\qquad\omega\longmapsto\left\langle x,R_{\infty}(\omega)y\right\rangle 
\]
are $\mathcal{F}_{n}$-measurable and $\mathcal{F}_{\infty}$-measurable,
respectively, where $\mathcal{F}_{\infty}=\sigma\left(\bigcup_{n\ge0}\mathcal{F}_{n}\right)$.
In particular, $\omega\mapsto\left\langle x,R_{\infty}(\omega)x\right\rangle $
is a nonnegative $\nu_{x}$-measurable function on $\Omega$. 
\end{prop}

\begin{proof}
Fix $x,y\in H$. For each $n$, the operator $R_{\omega|n}$ depends
only on the prefix $\omega|n$. There are only finitely many words
of length $n$, so $\omega\mapsto R_{\omega|n}$ is a simple, piecewise
constant map on the partition of $\Omega$ into length-$n$ cylinders.
It follows that 
\[
\omega\longmapsto\left\langle x,R_{\omega|n}y\right\rangle 
\]
is a simple $\mathcal{F}_{n}$-measurable function. By \prettyref{thm:B-1},
for each fixed $\omega$ we have strong convergence 
\[
R_{\omega|n}\xrightarrow[n\to\infty]{s}R_{\infty}(\omega).
\]
In particular, for fixed $x,y\in H$, 
\[
\left\langle x,R_{\omega|n}y\right\rangle \longrightarrow\left\langle x,R_{\infty}(\omega)y\right\rangle 
\]
pointwise in $\omega$. Since pointwise limits of measurable functions
are measurable, $\omega\mapsto\left\langle x,R_{\infty}(\omega)y\right\rangle $
is $\mathcal{F}_{\infty}$-measurable. 
\end{proof}
Thus the WR tree yields, for each initial state $x$, a well-defined
random boundary energy $\omega\mapsto\left\langle x,R_{\infty}(\omega)x\right\rangle $.
The next result shows that this random energy arises as the limit
of a natural supermartingale.

\subsection{Energy supermartingale and boundary limit}

Define the real-valued process 
\begin{equation}
M_{n}(\omega):=\left\langle x,R_{\omega|n}x\right\rangle ,\qquad n\ge0,\ \omega\in\Omega.\label{eq:d-1}
\end{equation}
By the monotonicity of the WR residuals along each branch, $\{M_{n}(\omega)\}_{n\ge0}$
is a nonincreasing sequence of nonnegative numbers for each fixed
$\omega$. 
\begin{thm}[Energy supermartingale]
 \label{thm:d-2} Let $x\in H$ with $\left\langle x,R_{0}x\right\rangle >0$,
and let $\nu_{x}$ be the WR path measure from \prettyref{thm:c-4}.
Then the process $\{M_{n}\}_{n\ge0}$ defined by \eqref{eq:d-1} is
a bounded nonnegative supermartingale with respect to $(\nu_{x},\mathcal{F}_{n})$.
In particular:
\begin{enumerate}
\item For each $n$, $M_{n}$ is $\mathcal{F}_{n}$-measurable and $0\le M_{n}(\omega)\le\left\langle x,R_{0}x\right\rangle $
for all $\omega\in\Omega$. 
\item For each $n$, 
\[
\mathbb{E}_{\nu_{x}}\left[M_{n+1}\mid\mathcal{F}_{n}\right]\le M_{n}(\omega)\quad\text{for }\nu_{x}\text{-almost every }\omega.
\]
\item There exists an integrable random variable $M_{\infty}$ such that
\[
M_{n}\longrightarrow M_{\infty}\quad\text{almost surely and in }L^{1}(\nu_{x}).
\]
Moreover, for $\nu_{x}$-almost every $\omega$, 
\[
M_{\infty}(\omega)=\left\langle x,R_{\infty}(\omega)x\right\rangle .
\]
\end{enumerate}
\end{thm}

\begin{proof}
(1) Measurability and boundedness follow from \prettyref{prop:d-1}
and from the operator inequalities $0\le R_{\omega|n}\le R_{0}$.

(2) Fix $n\ge0$ and a word $w$ of length $n$. On the cylinder $[w]$
we have $\omega|n=w$ and 
\[
M_{n}(\omega)=\left\langle x,R_{w}x\right\rangle .
\]
At the next step, the process moves to a child $wj$ with probability
$p_{x}(j\mid w)$, and the corresponding residual is $R_{wj}$. Therefore,
for $\omega\in[w]$, 
\[
\mathbb{E}_{\nu_{x}}\left[M_{n+1}\mid\mathcal{F}_{n}\right]=\sum_{j\in\mathcal{A}}p_{x}(j\mid w)\left\langle x,R_{wj}x\right\rangle .
\]
By \prettyref{thm:B-1} we have $0\le R_{wj}\le R_{w}$ for every
$j$, hence 
\[
\left\langle x,R_{wj}x\right\rangle \le\left\langle x,R_{w}x\right\rangle \quad\text{for all }j.
\]
Since $\sum_{j}p_{x}(j\mid w)=1$, we obtain 
\[
\sum_{j}p_{x}(j\mid w)\left\langle x,R_{wj}x\right\rangle \le\left\langle x,R_{w}x\right\rangle =M_{n}(\omega),
\]
for all $\omega\in[w]$. This proves the supermartingale inequality
\[
\mathbb{E}_{\nu_{x}}\left[M_{n+1}\mid\mathcal{F}_{n}\right]\le M_{n}\quad\text{almost surely.}
\]

(3) The process $\{M_{n}\}$ is nonnegative and uniformly bounded
by $\left\langle x,R_{0}x\right\rangle $, hence uniformly integrable.
By the supermartingale convergence theorem, there exists an integrable
random variable $M_{\infty}$ such that $M_{n}\to M_{\infty}$ almost
surely and in $L^{1}(\nu_{x})$. On the other hand, for each fixed
$\omega$, the sequence $\{R_{\omega|n}\}$ converges strongly to
$R_{\infty}(\omega)$ by \prettyref{thm:B-1} . In particular, for
our fixed $x$, 
\[
\left\langle x,R_{\omega|n}x\right\rangle \longrightarrow\left\langle x,R_{\infty}(\omega)x\right\rangle \quad\text{for every }\omega\in\Omega.
\]
Thus the pointwise limit of $M_{n}(\omega)$ is exactly $\left\langle x,R_{\infty}(\omega)x\right\rangle $.
Since almost sure limits are unique, it follows that 
\[
M_{\infty}(\omega)=\left\langle x,R_{\infty}(\omega)x\right\rangle \quad\text{for }\nu_{x}\text{-almost every }\omega.
\]
\end{proof}
\begin{rem}
The theorem shows that, along a $\nu_{x}$-random WR branch, the $x$-energy
of the residual operator decreases as a supermartingale and converges
almost surely to the boundary value $\left\langle x,R_{\infty}(\omega)x\right\rangle $.
In particular, the extinction event 
\[
E_{\mathrm{ext}}(x):=\left\{ \omega\in\Omega:\left\langle x,R_{\infty}(\omega)x\right\rangle =0\right\} 
\]
has probability 
\[
\nu_{x}(E_{\mathrm{ext}}(x))=\nu_{x}\left(\left\{ \omega:M_{\infty}(\omega)=0\right\} \right),
\]
and the complementary event describes random branches along which
a positive amount of $x$-energy survives at the boundary. This is
one natural sense in which the random WR boundary $\{R_{\infty}(\omega)\}$
encodes a  distribution of surviving energy in $B(H)_{+}$.
\end{rem}

\subsection{Expected energy balance along random branches}

Finally, we record an exact energy balance identity at the level of
expectations. For $\omega\in\Omega$ and $k\ge1$, let 
\[
\Delta_{k}(\omega):=D_{\omega|k-1,j_{k}},\qquad A_{k}(\omega):=\left\langle x,\Delta_{k}(\omega)x\right\rangle .
\]
By the pathwise telescoping identity, 
\[
R_{0}-R_{\infty}(\omega)=\sum^{\infty}_{k=1}\Delta_{k}(\omega)
\]
in the strong operator topology, and by taking inner products with
$x$ we obtain, for each $\omega$, 
\[
\left\langle x,R_{0}x\right\rangle =\left\langle x,R_{\infty}(\omega)x\right\rangle +\sum^{\infty}_{k=1}A_{k}(\omega),
\]
where the series converges monotonically. Applying the monotone convergence
theorem and \prettyref{thm:d-2}, we obtain:
\begin{prop}[Energy balance for random WR branches]
 \label{prop:d-4} Let $x\in H$ with $\left\langle x,R_{0}x\right\rangle >0$.
Then the random variables 
\[
A_{k}(\omega):=\left\langle x,\Delta_{k}(\omega)x\right\rangle ,\qquad k\ge1,
\]
are nonnegative and integrable, and 
\[
\left\langle x,R_{0}x\right\rangle =\mathbb{E}_{\nu_{x}}\left[\left\langle x,R_{\infty}(\omega)x\right\rangle \right]+\sum^{\infty}_{k=1}\mathbb{E}_{\nu_{x}}\left[A_{k}(\omega)\right].
\]
In particular, the expected total $x$-energy dissipated along a $\nu_{x}$-random
WR branch is finite and equal to 
\[
\sum^{\infty}_{k=1}\mathbb{E}_{\nu_{x}}\left[A_{k}(\omega)\right]=\left\langle x,R_{0}x\right\rangle -\mathbb{E}_{\nu_{x}}\left[\left\langle x,R_{\infty}(\omega)x\right\rangle \right].
\]
\end{prop}

\begin{proof}
For each $\omega$ the series $\sum_{k\ge1}A_{k}(\omega)$ converges
monotonically to $\left\langle x,R_{0}x\right\rangle -\left\langle x,R_{\infty}(\omega)x\right\rangle $,
so the pointwise sum is finite and bounded by $\left\langle x,R_{0}x\right\rangle $.
By monotone convergence and Fubin-Tonelli, 
\[
\mathbb{E}_{\nu_{x}}\left[\sum\nolimits^{\infty}_{k=1}A_{k}(\omega)\right]=\sum\nolimits^{\infty}_{k=1}\mathbb{E}_{\nu_{x}}\left[A_{k}(\omega)\right].
\]
Taking expectations in the pathwise identity 
\[
\left\langle x,R_{0}x\right\rangle =\left\langle x,R_{\infty}(\omega)x\right\rangle +\sum\nolimits^{\infty}_{k=1}A_{k}(\omega)
\]
and using \prettyref{thm:d-2} to identify $\left\langle x,R_{\infty}(\omega)x\right\rangle $
with the almost sure limit $M_{\infty}(\omega)$ gives the desired
equality. 
\end{proof}
\begin{cor}
Set 
\[
E_{\mathrm{ext}}\left(x\right):=\left\{ \omega\in\Omega:\left\langle x,R_{\infty}\left(\omega\right)x\right\rangle =0\right\} .
\]
Then the following are equivalent:
\begin{enumerate}
\item $\nu_{x}\left(E_{\mathrm{ext}}(x)\right)=1$. 
\item $E_{\nu_{x}}\left[\left\langle x,R_{\infty}\left(\omega\right)x\right\rangle \right]=0$. 
\item $\sum_{k\ge1}E_{\nu_{x}}\left[A_{k}\left(\omega\right)\right]=\left\langle x,R_{0}x\right\rangle $. 
\end{enumerate}
\end{cor}

Moreover, for every $n\ge0$ one has the tail identity 
\[
E_{\nu_{x}}\left[M_{n}\right]=E_{\nu_{x}}\left[\left\langle x,R_{\infty}\left(\omega\right)x\right\rangle \right]+\sum^{\infty}_{k=n+1}E_{\nu_{x}}\left[A_{k}\left(\omega\right)\right],
\]
and in particular 
\[
E_{\nu_{x}}\left[\left\langle x,R_{\infty}\left(\omega\right)x\right\rangle \right]=\inf_{n\ge0}E_{\nu_{x}}\left[M_{n}\right]=\lim_{n\to\infty}E_{\nu_{x}}\left[M_{n}\right].
\]

\begin{proof}
Since $\left\langle x,R_{\infty}\left(\omega\right)x\right\rangle \ge0$,
we have 
\[
\nu_{x}\left(E_{\mathrm{ext}}\left(x\right)\right)=1\Longleftrightarrow E_{\nu_{x}}\left[\left\langle x,R_{\infty}\left(\omega\right)x\right\rangle \right]=0.
\]
The equivalence of (2) and (3) is exactly \prettyref{prop:d-4}. 

For the tail identity, apply the pathwise telescoping identity at
depth $n$: 
\[
\left\langle x,R_{\omega|n}x\right\rangle =\left\langle x,R_{\infty}\left(\omega\right)x\right\rangle +\sum^{\infty}_{k=n+1}A_{k}\left(\omega\right),
\]
where the series converges monotonically, and take expectations, using
monotone convergence and Fubini-Tonelli as in \prettyref{prop:d-4}.
The final equalities follow because $E_{\nu_{x}}\left[M_{n}\right]$
is decreasing in $n$.
\end{proof}

\section{Uniform extinction}\label{sec:5}

In this section we show that, under a natural quantitative splitting
assumption on the projections $\{P_{j}\}$, the WR dynamics extinguish
the $x$-energy along $\nu_{x}$-almost every branch, and in fact
do so at an exponential rate in expectation. The result is dimension-free
and depends only on a uniform lower bound for the sum $\sum_{j}P_{j}$.

As before, we fix a vector $x\in H$ with $\left\langle x,R_{0}x\right\rangle >0$,
and we work with the WR tree $\{R_{w},D_{wj}\}$, the path measure
$\nu_{x}$, and the supermartingale 
\[
M_{n}(\omega):=\left\langle x,R_{\omega|n}x\right\rangle ,\qquad n\ge0,
\]
constructed in \prettyref{thm:d-2}. We also recall the step-$k$
dissipated operators 
\[
\Delta_{k}(\omega):=D_{\omega|k-1,j_{k}},\qquad A_{k}(\omega):=\left\langle x,\Delta_{k}(\omega)x\right\rangle ,\qquad k\ge1,
\]
where $j_{k}$ is the $k$-th letter of $\omega$. Recall that $ran(R_{w})$
are all contained in a fixed closed subspace $H_{0}\subset H$, where
$H_{0}=\overline{ran}(R^{1/2}_{0})$; see \eqref{eq:c-1}--\eqref{eq:c-2}.
\begin{assumption}
\label{assu:alpha} There exists a constant $\alpha>0$ such that
\begin{equation}
\sum^{m}_{j=1}P_{j}\ge\alpha I\quad\text{on }H_{0}.\label{eq:5-1}
\end{equation}
\end{assumption}

\begin{rem}
Note that \eqref{eq:5-1} holds with $\alpha=1$ if $\sum_{j}P_{j}=I$
on $H_{0}$, i.e., the splitting case \eqref{eq:c-3}. Since each
$P_{j}\le I$, we also have $\sum^{m}_{j=1}P_{j}\le mI$, so necessarily
$0<\alpha\le m$, and the quantity $c:=1-\frac{\alpha}{m}$ lies in
the interval $[0,1)$.

Equivalently, restricted to $H_{0}$, the subspaces $ran\left(P_{j}\right)$
with unit weights form a finite \textit{fusion frame} with bounds
$\alpha$ and $m$: 
\begin{equation}
\alpha\left\Vert x\right\Vert ^{2}\leq\sum^{m}_{j=1}\left\Vert P_{j}x\right\Vert ^{2}\leq m\left\Vert x\right\Vert ^{2},\qquad x\in H_{0}.\label{eq:f-2}
\end{equation}
We will nevertheless refer to \eqref{eq:5-1} as a leakage hypothesis,
since in the WR dynamics it measures the amount of residual energy
that is captured at each generation.

In particular, \eqref{eq:5-1} implies that $\overline{span}\left\{ ran\left(P_{j}\right)\right\} =H_{0}$.
If a direction in $H_{0}$ were orthogonal to every $ran\left(P_{j}\right)$,
then it would be invariant under the WR dynamics and extinction could
not hold in that direction. Thus some form of “no blind directions”
is necessary for the results of this section, and \eqref{eq:5-1}
is a quantitative version with a uniform margin $\alpha$.
\end{rem}

\prettyref{thm:5-3} shows that, under \eqref{eq:5-1}, the WR dynamics
must dissipate a fixed proportion of the $x$-energy at each level,
on average with respect to the path measure $\nu_{x}$. This leads
to uniform exponential decay of the expected residual energy, and
almost sure extinction along random branches. 
\begin{thm}
\label{thm:5-3} Assume \eqref{eq:5-1}. Then there exists a constant
\[
0\le c<1,\qquad c:=1-\frac{\alpha}{m},
\]
such that the following hold for every nonzero $x\in H$ with $\left\langle x,R_{0}x\right\rangle >0$:
\begin{enumerate}
\item For all $n\ge0$, 
\[
\mathbb{E}_{\nu_{x}}\left[M_{n+1}\right]\le c\,\mathbb{E}_{\nu_{x}}\left[M_{n}\right],
\]
and hence 
\[
\mathbb{E}_{\nu_{x}}\left[M_{n}\right]\le c^{n}\left\langle x,R_{0}x\right\rangle .
\]
\item Along $\nu_{x}$-almost every branch $\omega\in\Omega$, the residual
energy satisfies 
\[
\left\langle x,R_{\omega|n}x\right\rangle \longrightarrow0\quad\text{as }n\to\infty,
\]
that is, 
\[
\left\langle x,R_{\infty}(\omega)x\right\rangle =0\quad\text{for }\nu_{x}\text{-almost every }\omega.
\]
\end{enumerate}
In particular, the WR dynamics extinguish the $x$-energy almost surely
along a $\nu_{x}$-random branch, and the expected residual energy
decays at least at the rate $c^{n}$. 

\end{thm}

\begin{proof}
We first derive a uniform one-step contraction bound for the conditional
expectation of $M_{n+1}$ given $\mathcal{F}_{n}$. Fix $n\ge0$ and
a word $w$ of length $n$. On the cylinder $[w]$ we have $\omega|n=w$
and $M_{n}(\omega)=\left\langle x,R_{w}x\right\rangle $. Set 
\[
v:=R^{1/2}_{w}x\in H_{0}.
\]
Then $\left\langle x,R_{w}x\right\rangle =\left\Vert v\right\Vert ^{2}$.
For each child $wj$, the next residual is 
\[
R_{wj}=R^{1/2}_{w}(I-P_{j})R^{1/2}_{w},
\]
so 
\[
\left\langle x,R_{wj}x\right\rangle =\left\langle v,(I-P_{j})v\right\rangle =\left\Vert v\right\Vert ^{2}-\left\Vert P_{j}v\right\Vert ^{2}.
\]
Next, recall that the energy-biased transition probabilities at node
$w$ are given by 
\[
p_{x}(j\mid w):=\frac{a_{wj}}{\sum_{k}a_{wk}}=\frac{\left\langle x,D_{wj}x\right\rangle }{\sum_{k}\left\langle x,D_{wk}x\right\rangle }=\frac{\left\Vert P_{j}v\right\Vert ^{2}}{\sum_{k}\left\Vert P_{k}v\right\Vert ^{2}},
\]
whenever $\sum_{k}a_{wk}>0$, and by our convention they are given
by some fixed probability vector when $\sum_{k}a_{wk}=0$. In either
case, $p_{x}(\cdot\mid w)$ is a probability vector on $\mathcal{A}$.

Denote 
\[
s_{j}:=\left\Vert P_{j}v\right\Vert ^{2},\qquad S:=\sum^{m}_{j=1}s_{j}.
\]
Then 
\[
p_{x}(j\mid w)=\begin{cases}
{\displaystyle \frac{s_{j}}{S},} & \text{if }S>0,\\
\text{a fixed }q_{j}, & \text{if }S=0.
\end{cases}
\]
If $S=0$, then $P_{j}v=0$ for all $j$, and hence 
\[
0=\sum\nolimits^{m}_{j=1}\left\Vert P_{j}v\right\Vert ^{2}=\left\langle v,\left(\sum\nolimits^{m}_{j=1}P_{j}\right)v\right\rangle \ge\alpha\left\Vert v\right\Vert ^{2}
\]
by \eqref{eq:5-1}. Since $\alpha>0$, this forces $v=0$, so $\left\Vert v\right\Vert ^{2}=0$
and therefore $M_{n}(\omega)=0$ on $[w]$. In particular, 
\[
\mathbb{E}_{\nu_{x}}\left[M_{n+1}\mid\mathcal{F}_{n}\right]=0=c\,M_{n}(\omega)
\]
for all $\omega\in[w]$, and the desired inequality holds trivially
on that cylinder.

If $S>0$, then 
\[
\mathbb{E}_{\nu_{x}}\left[M_{n+1}\mid\mathcal{F}_{n}\right]=\sum\nolimits^{m}_{j=1}p_{x}(j\mid w)\left\langle x,R_{wj}x\right\rangle =\sum\nolimits^{m}_{j=1}\frac{s_{j}}{S}\left(\left\Vert v\right\Vert ^{2}-s_{j}\right).
\]
Dividing and rearranging, we obtain 
\[
\frac{\mathbb{E}_{\nu_{x}}\left[M_{n+1}\mid\mathcal{F}_{n}\right]}{\left\Vert v\right\Vert ^{2}}=\sum\nolimits^{m}_{j=1}\frac{s_{j}}{S}\left(1-\frac{s_{j}}{\left\Vert v\right\Vert ^{2}}\right)=1-\frac{\sum^{m}_{j=1}s^{2}_{j}}{S\left\Vert v\right\Vert ^{2}}.
\]
By the Cauchy--Schwarz inequality, 
\[
\left(\sum\nolimits^{m}_{j=1}s_{j}\right)^{2}\le m\sum\nolimits^{m}_{j=1}s^{2}_{j},
\]
so 
\[
\sum\nolimits^{m}_{j=1}s^{2}_{j}\ge\frac{S^{2}}{m}.
\]
Therefore, 
\[
\frac{\mathbb{E}_{\nu_{x}}\left[M_{n+1}\mid\mathcal{F}_{n}\right]}{\left\Vert v\right\Vert ^{2}}\le1-\frac{S^{2}/m}{S\left\Vert v\right\Vert ^{2}}=1-\frac{S}{m\left\Vert v\right\Vert ^{2}}.
\]
On the other hand, \eqref{eq:5-1} applied to $v\in H_{0}$ yields
\[
S=\sum\nolimits^{m}_{j=1}\left\Vert P_{j}v\right\Vert ^{2}=\left\langle v,\left(\sum\nolimits^{m}_{j=1}P_{j}\right)v\right\rangle \ge\alpha\left\Vert v\right\Vert ^{2}.
\]
Combining the inequalities, we obtain 
\[
\frac{\mathbb{E}_{\nu_{x}}\left[M_{n+1}\mid\mathcal{F}_{n}\right]}{\left\Vert v\right\Vert ^{2}}\le1-\frac{\alpha}{m},
\]
that is, 
\[
\mathbb{E}_{\nu_{x}}\left[M_{n+1}\mid\mathcal{F}_{n}\right]\le\left(1-\frac{\alpha}{m}\right)\left\Vert v\right\Vert ^{2}=\left(1-\frac{\alpha}{m}\right)M_{n}(\omega).
\]
Summarizing, we have shown that for every $n$ and every cylinder
$[w]$ of length $n$, 
\[
\mathbb{E}_{\nu_{x}}\left[M_{n+1}\mid\mathcal{F}_{n}\right]\le c\,M_{n}\quad\text{almost surely,}
\]
with $c=1-\alpha/m\in[0,1)$.

Taking expectations in this inequality yields 
\[
\mathbb{E}_{\nu_{x}}\left[M_{n+1}\right]\le c\,\mathbb{E}_{\nu_{x}}\left[M_{n}\right]\quad\text{for all }n\ge0.
\]
Iterating gives 
\[
\mathbb{E}_{\nu_{x}}\left[M_{n}\right]\le c^{n}\,\mathbb{E}_{\nu_{x}}\left[M_{0}\right]=c^{n}\left\langle x,R_{0}x\right\rangle ,
\]
which proves part (1).

For part (2), recall that $\{M_{n}\}$ is a bounded nonnegative supermartingale.
By \prettyref{thm:d-2}, there exists an integrable random variable
$M_{\infty}$ such that 
\[
M_{n}\longrightarrow M_{\infty}\quad\text{almost surely and in }L^{1}(\nu_{x}),
\]
and for $\nu_{x}$-almost every $\omega$, 
\[
M_{\infty}(\omega)=\left\langle x,R_{\infty}(\omega)x\right\rangle .
\]
On the other hand, from part (1) we have 
\[
\mathbb{E}_{\nu_{x}}\left[M_{n}\right]\le c^{n}\left\langle x,R_{0}x\right\rangle \longrightarrow0.
\]
Since $M_{n}\to M_{\infty}$ in $L^{1}(\nu_{x})$, this forces 
\[
\mathbb{E}_{\nu_{x}}\left[M_{\infty}\right]=0.
\]
Because $M_{\infty}\ge0$, we conclude that $M_{\infty}(\omega)=0$
for $\nu_{x}$-almost every $\omega$. Equivalently, 
\[
\left\langle x,R_{\infty}(\omega)x\right\rangle =0\quad\text{for }\nu_{x}\text{-almost every }\omega.
\]
\end{proof}
\begin{cor}
The constant $c=1-\alpha/m$ is dimension-free and depends only on
the parameter $\alpha$ and the branching number $m$. In the splitting
case, $\sum_{j}P_{j}=I$ on $H_{0}$, one has $\alpha=1$ and therefore
\[
\mathbb{E}_{\nu_{x}}\left[M_{n}\right]\le\left(1-\frac{1}{m}\right)^{n}\left\langle x,R_{0}x\right\rangle .
\]
Thus the expected residual energy decays at least at rate $(1-1/m)^{n}$
for every initial state $x$, and almost sure extinction holds along
$\nu_{x}$-random branches.
\end{cor}

\begin{rem}
Assumption \eqref{eq:5-1} can be understood as a quantitative nondegeneracy
of the projections $\{P_{j}\}$ on the WR-orbit of $x$. For each
residual vector $v=R^{1/2}_{w}x$, the inequality 
\[
\sum_{j}\left\Vert P_{j}v\right\Vert ^{2}\ge\alpha\left\Vert v\right\Vert ^{2}
\]
says that, on average, the family $\{P_{j}\}$ captures a fixed fraction
of the energy of $v$. The theorem then asserts that this uniform
leakage propagates along the WR tree and forces extinction of the
$x$-energy at the boundary, with a quantitative rate controlled solely
by $\alpha$ and $m$. Thus, an operator-theoretic splitting condition
translates into a probabilistic extinction for the WR dynamics. 
\end{rem}

\section{Random atomic expansions along a branch}\label{sec:6}

In this section we show that, under the leakage assumption \eqref{eq:5-1},
the WR dynamics generate a random atomic (frame-like) expansion of
the quadratic form of $R_{0}$ along typical branches. For each fixed
state $x\in H$ with $\left\langle x,R_{0}x\right\rangle >0$, the
stepwise dissipated operators $\Delta_{k}(\omega)$ can be decomposed
spectrally, and the resulting rank-one pieces combine to give a pathwise
Parseval-type identity 
\[
\left\langle x,R_{0}x\right\rangle =\sum_{k,r}\left|\left\langle x,\varphi_{k,r}\left(\omega\right)\right\rangle \right|^{2}
\]
for $\nu_{x}$-almost every branch $\omega$. This gives an operator-theoretic
instance of a ``WR frame on the boundary'', tied to a single initial
state $x$.

We continue with the setting of Sections \ref{sec:4} and \ref{sec:5}.
In particular: 
\begin{itemize}
\item $R_{w}$ and $D_{wj}$ are the WR residuals and dissipated pieces
at node $w$. 
\item For $\omega\in\Omega$ and $k\ge1$, the step-$k$ dissipated operator
along $\omega$ is 
\[
\Delta_{k}(\omega):=D_{\omega|k-1,j_{k}},
\]
where $j_{k}$ is the $k$-th letter of $\omega$. 
\item For a fixed state $x\in H$, the scalar dissipated energy at step
$k$ is 
\[
A_{k}(\omega):=\left\langle x,\Delta_{k}(\omega)x\right\rangle .
\]
\end{itemize}
The pathwise telescoping identity reads 
\[
\left\langle x,R_{0}x\right\rangle =\left\langle x,R_{\infty}(\omega)x\right\rangle +\sum^{\infty}_{k=1}A_{k}(\omega),
\]
with monotone convergence in $k$. Under \eqref{eq:5-1}, we have
\[
\left\langle x,R_{\infty}(\omega)x\right\rangle =0\quad\text{for }\nu_{x}\text{-almost every }\omega
\]
by \prettyref{thm:5-3}. 

For this section we impose an additional condition:
\begin{assumption}
\label{assump:trace} The initial operator $R_{0}$ is trace class. 
\end{assumption}

Since $0\le R_{w}\le R_{0}$ for every $w$, this implies that each
$R_{w}$ is trace class. In particular, for every $\omega\in\Omega$
and $k\ge1$, 
\[
\Delta_{k}(\omega)=R^{1/2}_{\omega|k-1}P_{j_{k}}R^{1/2}_{\omega|k-1}
\]
is trace class. We now pass from operator-valued dissipation $\Delta_{k}(\omega)$
to rank-one atoms.

\subsection{Spectral atoms of the dissipated operators}

Fix a branch $\omega\in\Omega$ and a step $k\ge1$. By Assumption
\ref{assump:trace}, the operator $\Delta_{k}(\omega)$ is positive
and compact. By the spectral theorem for compact selfadjoint operators,
there exists an orthonormal family $\{u_{k,r}(\omega)\}_{r\in\mathcal{R}_{k}(\omega)}\subset H$
and nonnegative eigenvalues $\{\lambda_{k,r}(\omega)\}_{r\in\mathcal{R}_{k}(\omega)}$
such that 
\[
\Delta_{k}(\omega)=\sum_{r\in\mathcal{R}_{k}(\omega)}\lambda_{k,r}(\omega)\left|u_{k,r}(\omega)\right\rangle \left\langle u_{k,r}(\omega)\right|
\]
with convergence in the strong operator topology. Here $\mathcal{R}_{k}(\omega)$
is at most countable, assuming $H$ separable. We define the corresponding
rank-one ``atoms'' 
\[
\varphi_{k,r}(\omega):=\lambda_{k,r}(\omega)^{1/2}u_{k,r}(\omega),\qquad k\ge1,\ r\in\mathcal{R}_{k}(\omega).
\]
Then 
\[
\Delta_{k}(\omega)=\sum_{r\in\mathcal{R}_{k}(\omega)}\left|\varphi_{k,r}(\omega)\right\rangle \left\langle \varphi_{k,r}(\omega)\right|.
\]
For any $x\in H$ we therefore have 
\[
A_{k}(\omega)=\left\langle x,\Delta_{k}(\omega)x\right\rangle =\sum_{r\in\mathcal{R}_{k}(\omega)}\left|\left\langle x,\varphi_{k,r}(\omega)\right\rangle \right|^{2},
\]
with the scalar series converging monotonically in $r$.

\subsection{A random WR Parseval expansion for a fixed state}

We now combine the spectral decomposition of $\Delta_{k}(\omega)$
with the extinction result of \prettyref{sec:5} to obtain a pathwise
expansion of $\left\langle x,R_{0}x\right\rangle $ in terms of the
atoms $\varphi_{k,r}(\omega)$.
\begin{thm}
\label{thm:6-2} Assume \eqref{eq:5-1} and $R_{0}$ is trace class,
and fix a nonzero state $x\in H$ with $\left\langle x,R_{0}x\right\rangle >0$.
Then there exists a $\nu_{x}$-full measure set $\Omega_{x}\subset\Omega$
such that for every $\omega\in\Omega_{x}$, 
\[
\left\langle x,R_{0}x\right\rangle =\sum^{\infty}_{k=1}A_{k}(\omega)=\sum^{\infty}_{k=1}\sum_{r\in\mathcal{R}_{k}(\omega)}\left|\left\langle x,\varphi_{k,r}(\omega)\right\rangle \right|^{2},
\]
where the double series converges monotonically. In particular, for
$\nu_{x}$-almost every branch $\omega$, the family 
\[
\left\{ \varphi_{k,r}(\omega):k\ge1,\ r\in\mathcal{R}_{k}(\omega)\right\} 
\]
yields a Parseval-type atomic decomposition of the scalar quantity
$\left\langle x,R_{0}x\right\rangle $. 
\end{thm}

\begin{proof}
Fix $x\in H$ with $\left\langle x,R_{0}x\right\rangle >0$. By the
pathwise telescoping identity and the definition of $A_{k}$, for
every $\omega\in\Omega$, 
\[
\left\langle x,R_{0}x\right\rangle =\left\langle x,R_{\infty}(\omega)x\right\rangle +\sum^{\infty}_{k=1}A_{k}(\omega),
\]
with $\sum_{k\ge1}A_{k}(\omega)$ converging monotonically and bounded
above by $\left\langle x,R_{0}x\right\rangle $. By \prettyref{thm:5-3}
(under \eqref{eq:5-1}) we know that 
\[
\left\langle x,R_{\infty}(\omega)x\right\rangle =0\quad\text{for }\nu_{x}\text{-almost every }\omega.
\]
Let $\Omega_{x}\subset\Omega$ denote the full-measure subset on which
this holds. For each $\omega\in\Omega_{x}$ the telescoping identity
reduces to 
\[
\left\langle x,R_{0}x\right\rangle =\sum^{\infty}_{k=1}A_{k}(\omega),
\]
with monotone convergence in $k$. Next, for each $\omega$ and $k$
we expand $A_{k}(\omega)$ using the spectral decomposition of $\Delta_{k}(\omega)$:
\[
A_{k}(\omega)=\left\langle x,\Delta_{k}(\omega)x\right\rangle =\sum_{r\in\mathcal{R}_{k}(\omega)}\left|\left\langle x,\varphi_{k,r}(\omega)\right\rangle \right|^{2},
\]
with monotone convergence in $r$. Combining these relations, we obtain
\[
\left\langle x,R_{0}x\right\rangle =\sum^{\infty}_{k=1}A_{k}(\omega)=\sum^{\infty}_{k=1}\sum_{r\in\mathcal{R}_{k}(\omega)}\left|\left\langle x,\varphi_{k,r}(\omega)\right\rangle \right|^{2}
\]
for all $\omega\in\Omega_{x}$. Since all terms are nonnegative, the
double series converges monotonically to $\left\langle x,R_{0}x\right\rangle $,
and the order of summation may be rearranged arbitrarily if desired. 
\end{proof}
\begin{rem}
For a fixed $x$, \prettyref{thm:6-2} shows that along $\nu_{x}$-almost
every branch $\omega$ the entire $x$-energy $\left\langle x,R_{0}x\right\rangle $
can be written as an $\ell^{2}$-type sum of coefficients 
\[
\left\langle x,\varphi_{k,r}(\omega)\right\rangle 
\]
against a random countable family of atoms $\{\varphi_{k,r}(\omega)\}$
generated by the WR dynamics. In this sense, each typical WR branch
carries a ``Parseval frame'' for the one-dimensional ray $\mathbb{C}x$,
with frame elements living entirely inside the cone $B(H)_{+}$ through
their rank-one outer products. The indexing set 
\[
\left\{ (k,r):k\ge1,\ r\in\mathcal{R}_{k}(\omega)\right\} 
\]
inherits the tree-like, branching structure of the WR energy tree,
and in concrete finite-dimensional examples this random family exhibits
a visibly fractal organization in $B(H)_{+}$. 
\end{rem}

\begin{rem}
One can extend this picture to a finite or countable family of initial
states $\{x_{\alpha}\}\subset H$. For each $x_{\alpha}$ one constructs
the corresponding path measure $\nu_{x_{\alpha}}$ and obtains a full-measure
set $\Omega_{x_{\alpha}}$ on which 
\[
\left\langle x_{\alpha},R_{0}x_{\alpha}\right\rangle =\sum_{k,r}\left|\left\langle x_{\alpha},\varphi_{k,r}(\omega)\right\rangle \right|^{2}.
\]
The operators $\Delta_{k}(\omega)$ and their spectral atoms $\varphi_{k,r}(\omega)$
are common to all $\alpha$; only the sampling measure $\nu_{x_{\alpha}}$
on $\Omega$ changes with the initial state. Thus the WR tree defines
a single family of atoms in $H$, while different initial states probe
this family through different energy-biased random walks on the tree.
\end{rem}

\section{Parseval-type frame via trace-biased branching}\label{sec:7}

In this section we construct a second path measure, now trace-biased
rather than $x$-energy biased, and prove that under \eqref{eq:5-1}
and Assumption \ref{assump:trace}, the WR dynamics along almost every
branch (with respect to this new measure) generate a countable atom
system whose rank-one outer products sum strongly to $R_{0}$. Equivalently,
along almost every branch we obtain a Parseval-type frame identity
with frame operator $R_{0}$ on the energy space $\overline{ran}(R^{1/2}_{0})$.

Throughout this section we assume \eqref{eq:5-1} and that $R_{0}$
is  trace class.

\subsection{Trace-biased transition probabilities and path measure}

For each node $w\in\mathcal{W}$ define the trace-dissipated weights
\[
s^{\mathrm{tr}}_{wj}:=\mathrm{tr}(D_{wj})=\mathrm{tr}\left(R^{1/2}_{w}P_{j}R^{1/2}_{w}\right)\ge0,\qquad j\in\mathcal{A}.
\]
Set 
\[
S^{\mathrm{tr}}_{w}:=\sum^{m}_{j=1}s^{\mathrm{tr}}_{wj}=\mathrm{tr}\left(R^{1/2}_{w}\left(\sum\nolimits^{m}_{j=1}P_{j}\right)R^{1/2}_{w}\right).
\]

We say that $w$ is trace-alive if $S^{\mathrm{tr}}_{w}>0$ and trace-dead
otherwise. On trace-alive nodes we define 
\[
p^{\mathrm{tr}}(j\mid w):=\frac{s^{\mathrm{tr}}_{wj}}{\sum^{m}_{k=1}s^{\mathrm{tr}}_{wk}},\qquad j\in\mathcal{A}.
\]
On trace-dead nodes we fix once and for all a reference probability
vector 
\[
q=(q_{1},\dots,q_{m})\in(0,1)^{m},\qquad\sum^{m}_{j=1}q_{j}=1,
\]
and set 
\[
p^{\mathrm{tr}}(j\mid w):=q_{j},\qquad j\in\mathcal{A}
\]
whenever $S^{\mathrm{tr}}_{w}=0$. Thus in all cases $p^{\mathrm{tr}}(\cdot\mid w)$
is a probability distribution on $\mathcal{A}$.

For a finite word $w=j_{1}\cdots j_{n}$ we define the cylinder weight
\[
\nu^{\mathrm{tr}}\left([w]\right):=p^{\mathrm{tr}}(j_{1}\mid\emptyset)p^{\mathrm{tr}}(j_{2}\mid j_{1})\cdots p^{\mathrm{tr}}(j_{n}\mid j_{1}\cdots j_{n-1}).
\]
As in \prettyref{sec:3}, this gives a consistent family of weights
on cylinders: 
\[
\nu^{\mathrm{tr}}\left([w]\right)=\sum^{m}_{j=1}\nu^{\mathrm{tr}}\left([wj]\right),\qquad w\in\mathcal{W},
\]
and the Kolmogorov--Caratheodory extension theorem yields a unique
Borel probability measure $\nu^{\mathrm{tr}}$ on $\Omega=\mathcal{A}^{\mathbb{N}}$
that extends these cylinder weights.

We call $\nu^{\mathrm{tr}}$ the trace-biased WR path measure. Intuitively,
at each trace-alive node $w$ the next index $j$ is chosen with probability
proportional to the trace $\mathrm{tr}(D_{wj})$ dissipated by channel
$j$; when no trace is dissipated the branching is governed by the
fixed reference vector $q$.

\subsection{A trace supermartingale and extinction of the trace}

We now construct the trace analogue of the energy supermartingale
from \prettyref{sec:4}.

For $\omega\in\Omega$ and $n\ge0$ set 
\[
T_{n}(\omega):=\mathrm{tr}(R_{\omega|n}).
\]
Since $R_{\omega|n}$ is positive and trace class for all $\omega$
and $n$, $T_{n}(\omega)\ge0$, and by $\mathrm{tr}(R_{\omega|n})\le\mathrm{tr}(R_{0})$
we have 
\[
0\le T_{n}(\omega)\le\mathrm{tr}(R_{0}),\qquad\omega\in\Omega,\ n\ge0.
\]
Let $\{\mathcal{F}_{n}\}$ be the canonical filtration generated by
the prefixes $\omega|n$.
\begin{thm}
\label{thm:7-1} Assume \eqref{eq:5-1} and that $R_{0}$ is trace
class. Then there exists a constant 
\[
0\le c<1,\qquad c:=1-\frac{\alpha}{m},
\]
such that the following hold:
\begin{enumerate}
\item $\{T_{n}\}_{n\ge0}$ is a bounded nonnegative supermartingale on $(\Omega,\nu^{\mathrm{tr}},\mathcal{F}_{n})$,
and for all $n\ge0$, 
\[
\mathbb{E}_{\nu^{\mathrm{tr}}}\left[T_{n+1}\right]\le c\,\mathbb{E}_{\nu^{\mathrm{tr}}}\left[T_{n}\right].
\]
\item For all $n\ge0$, 
\[
\mathbb{E}_{\nu^{\mathrm{tr}}}\left[T_{n}\right]\le c^{n}\mathrm{tr}(R_{0}),
\]
and the limit 
\[
T_{\infty}(\omega):=\lim_{n\to\infty}T_{n}(\omega)=\mathrm{tr}(R_{\infty}(\omega))
\]
exists for $\nu^{\mathrm{tr}}$-almost every $\omega$.
\item For $\nu^{\mathrm{tr}}$-almost every $\omega$, 
\[
\mathrm{tr}(R_{\infty}(\omega))=0,
\]
and therefore $R_{\infty}(\omega)=0$ as an operator. 
\end{enumerate}
\end{thm}

\begin{proof}
We first show the one-step contraction inequality for the conditional
expectation.

Fix $n\ge0$ and a word $w$ of length $n$. On the cylinder $[w]$
we have $\omega|n=w$ and 
\[
T_{n}(\omega)=\mathrm{tr}(R_{w}).
\]
Set 
\[
A:=\mathrm{tr}(R_{w}),\qquad s_{j}:=\mathrm{tr}(D_{wj})=\mathrm{tr}\left(R^{1/2}_{w}P_{j}R^{1/2}_{w}\right)\ge0,\qquad S:=\sum^{m}_{j=1}s_{j}.
\]
By construction, 
\[
p^{\mathrm{tr}}(j\mid w)=\begin{cases}
{\displaystyle \frac{s_{j}}{S},} & \text{if }S>0,\\
q_{j}, & \text{if }S=0.
\end{cases}
\]

We consider the two cases separately.

If $S=0$, then $s_{j}=0$ for all $j$. In particular 
\[
0=\sum^{m}_{j=1}\mathrm{tr}(R^{1/2}_{w}P_{j}R^{1/2}_{w})=\mathrm{tr}\left(R^{1/2}_{w}\left(\sum\nolimits^{m}_{j=1}P_{j}\right)R^{1/2}_{w}\right)\ge\alpha\,\mathrm{tr}(R_{w}),
\]
by \eqref{eq:5-1} on $H_{0}$. Since $\alpha>0$, it follows that
$\mathrm{tr}(R_{w})=0$, so $A=0$ and hence $T_{n}(\omega)=0$ on
$[w]$. Because $R_{w}\ge0$ is trace class, $\mathrm{tr}(R_{w})=0$
implies $R_{w}=0$, and therefore 
\[
R_{wj}=R^{1/2}_{w}(I-P_{j})R^{1/2}_{w}=0,\qquad\mathrm{tr}(R_{wj})=0
\]
for all $j$. Thus 
\[
\mathbb{E}_{\nu^{\mathrm{tr}}}\left[T_{n+1}\mid\mathcal{F}_{n}\right](\omega)=\sum^{m}_{j=1}p^{\mathrm{tr}}(j\mid w)\,\mathrm{tr}(R_{wj})=0=c\,T_{n}(\omega)
\]
on $[w]$. In particular, the desired inequality holds with equality
on trace-dead cylinders.

If $S>0$, then on $[w]$ we have 
\[
\mathbb{E}_{\nu^{\mathrm{tr}}}\left[T_{n+1}\mid\mathcal{F}_{n}\right](\omega)=\sum^{m}_{j=1}p^{\mathrm{tr}}(j\mid w)\,\mathrm{tr}(R_{wj}).
\]
As in \prettyref{sec:5}, 
\[
R_{wj}=R^{1/2}_{w}(I-P_{j})R^{1/2}_{w},
\]
so 
\[
\mathrm{tr}(R_{wj})=\mathrm{tr}\left(R^{1/2}_{w}(I-P_{j})R^{1/2}_{w}\right)=\mathrm{tr}(R_{w})-\mathrm{tr}\left(R^{1/2}_{w}P_{j}R^{1/2}_{w}\right)=A-s_{j}.
\]
Therefore 
\[
\frac{\mathbb{E}_{\nu^{\mathrm{tr}}}[T_{n+1}\mid\mathcal{F}_{n}](\omega)}{A}=\sum^{m}_{j=1}\frac{s_{j}}{S}\left(1-\frac{s_{j}}{A}\right)=1-\frac{\sum^{m}_{j=1}s^{2}_{j}}{SA}.
\]

By Cauchy--Schwarz, 
\[
\left(\sum\nolimits^{m}_{j=1}s_{j}\right)^{2}\le m\sum\nolimits^{m}_{j=1}s^{2}_{j},
\]
so 
\[
\sum^{m}_{j=1}s^{2}_{j}\ge\frac{S^{2}}{m}.
\]
Thus 
\[
\frac{\mathbb{E}_{\nu^{\mathrm{tr}}}[T_{n+1}\mid\mathcal{F}_{n}](\omega)}{A}\le1-\frac{S^{2}/m}{SA}=1-\frac{S}{mA}.
\]

On the other hand, applying \eqref{eq:5-1} to $R^{1/2}_{w}$ gives
\[
S=\sum^{m}_{j=1}\mathrm{tr}\left(R^{1/2}_{w}P_{j}R^{1/2}_{w}\right)=\mathrm{tr}\left(R^{1/2}_{w}\left(\sum\nolimits^{m}_{j=1}P_{j}\right)R^{1/2}_{w}\right)\ge\alpha\,\mathrm{tr}(R_{w})=\alpha A.
\]
Consequently, 
\[
\frac{\mathbb{E}_{\nu^{\mathrm{tr}}}[T_{n+1}\mid\mathcal{F}_{n}](\omega)}{A}\le1-\frac{\alpha}{m},
\]
or equivalently, 
\[
\mathbb{E}_{\nu^{\mathrm{tr}}}\left[T_{n+1}\mid\mathcal{F}_{n}\right](\omega)\le\left(1-\frac{\alpha}{m}\right)A=\left(1-\frac{\alpha}{m}\right)T_{n}(\omega).
\]

Thus, we have shown that for every $n$ and every cylinder $[w]$
of length $n$, 
\[
\mathbb{E}_{\nu^{\mathrm{tr}}}\left[T_{n+1}\mid\mathcal{F}_{n}\right]\le c\,T_{n}\quad\text{almost surely,}
\]
with $c=1-\alpha/m\in[0,1)$. This proves that $\{T_{n}\}$ is a bounded
nonnegative supermartingale and yields part (1) upon taking expectations:
\[
\mathbb{E}_{\nu^{\mathrm{tr}}}\left[T_{n+1}\right]\le c\,\mathbb{E}_{\nu^{\mathrm{tr}}}\left[T_{n}\right]\quad\text{for all }n\ge0.
\]

Iterating gives 
\[
\mathbb{E}_{\nu^{\mathrm{tr}}}\left[T_{n}\right]\le c^{n}\mathbb{E}_{\nu^{\mathrm{tr}}}\left[T_{0}\right]=c^{n}\mathrm{tr}(R_{0}),
\]
which is part (2) for the expectations.

Since $\{T_{n}\}$ is a bounded nonnegative supermartingale, the supermartingale
convergence theorem (\prettyref{thm:d-2}) implies that there exists
an integrable random variable $T_{\infty}$ such that 
\[
T_{n}\longrightarrow T_{\infty}\quad\text{almost surely and in }L^{1}(\nu^{\mathrm{tr}}).
\]
On the other hand, for each fixed $\omega$, the sequence $R_{\omega|n}$
is monotone decreasing in the Loewner order and bounded below by $0$,
so it converges strongly to a positive operator $R_{\infty}(\omega)$.
By monotone convergence of the trace, 
\[
T_{n}(\omega)=\mathrm{tr}(R_{\omega|n})\downarrow\mathrm{tr}(R_{\infty}(\omega)),
\]
so 
\[
T_{\infty}(\omega)=\mathrm{tr}(R_{\infty}(\omega))
\]
for $\nu^{\mathrm{tr}}$-almost every $\omega$.

From the bound 
\[
\mathbb{E}_{\nu^{\mathrm{tr}}}\left[T_{n}\right]\le c^{n}\mathrm{tr}(R_{0})\longrightarrow0
\]
and the $L^{1}$-convergence $T_{n}\to T_{\infty}$, we deduce 
\[
\mathbb{E}_{\nu^{\mathrm{tr}}}\left[T_{\infty}\right]=0.
\]
Since $T_{\infty}\ge0$, this forces $T_{\infty}(\omega)=0$ for $\nu^{\mathrm{tr}}$-almost
every $\omega$, that is, 
\[
\mathrm{tr}(R_{\infty}(\omega))=0\quad\text{for }\nu^{\mathrm{tr}}\text{-almost every }\omega.
\]

Finally, each $R_{\infty}(\omega)$ is positive and trace class, so
$\mathrm{tr}(R_{\infty}(\omega))=0$ implies $R_{\infty}(\omega)=0$.
This proves part (3). 
\end{proof}
\begin{rem}
This theorem is a trace analogue of the extinction result for $\langle x,R_{\omega|n}x\rangle$
in \prettyref{thm:5-3}. The difference is that $\nu^{\mathrm{tr}}$
depends only on the WR tree and the trace functional, not on a particular
vector $x$, and the conclusion $R_{\infty}(\omega)=0$ is an operator
statement valid along $\nu^{\mathrm{tr}}$-almost every branch. 
\end{rem}

\subsection{A global random WR Parseval-type frame}

We now combine \prettyref{thm:7-1} with the spectral decomposition
of the dissipated operators $\Delta_{k}(\omega)$ from \prettyref{sec:6}
to obtain an operator-level random Parseval-type frame identity.

Recall that $R_{0}$ assumed to be of trace class, and hence each
$R_{w}$ and $\Delta_{k}(\omega)$ is trace class and compact. For
each $\omega\in\Omega$ and $k\ge1$, the spectral theorem for compact
selfadjoint operators gives an orthonormal family $\{u_{k,r}(\omega)\}_{r\in\mathcal{R}_{k}(\omega)}\subset H$
and eigenvalues $\lambda_{k,r}(\omega)\ge0$ such that 
\[
\Delta_{k}(\omega)=\sum_{r\in\mathcal{R}_{k}(\omega)}\lambda_{k,r}(\omega)\left|u_{k,r}(\omega)\right\rangle \left\langle u_{k,r}(\omega)\right|,
\]
with convergence in the strong operator topology, where $\mathcal{R}_{k}(\omega)$
is at most countable. We then define the atoms 
\[
\varphi_{k,r}(\omega):=\lambda_{k,r}(\omega)^{1/2}u_{k,r}(\omega),
\]
so that 
\[
\Delta_{k}(\omega)=\sum_{r\in\mathcal{R}_{k}(\omega)}\left|\varphi_{k,r}(\omega)\right\rangle \left\langle \varphi_{k,r}(\omega)\right|.
\]

For any $x\in H$, 
\[
\left\langle x,\Delta_{k}(\omega)x\right\rangle =\sum_{r\in\mathcal{R}_{k}(\omega)}\left|\left\langle x,\varphi_{k,r}(\omega)\right\rangle \right|^{2},
\]
with the scalar series converging monotonically in $r$.

We can now state the main result of this section.
\begin{thm}
\label{thm:7-3} Assume \eqref{eq:5-1} and that $R_{0}$ is trace
class. Then there exists a $\nu^{\mathrm{tr}}$-full measure set $\Omega_{0}\subset\Omega$
such that for every $\omega\in\Omega_{0}$ the following hold:
\begin{enumerate}
\item The operator-level telescoping identity 
\[
R_{0}=\sum^{\infty}_{k=1}\Delta_{k}(\omega)
\]
holds in the strong operator topology.
\item For every $x\in H$, 
\[
\left\langle x,R_{0}x\right\rangle =\sum^{\infty}_{k=1}\left\langle x,\Delta_{k}(\omega)x\right\rangle =\sum^{\infty}_{k=1}\sum_{r\in\mathcal{R}_{k}(\omega)}\left|\left\langle x,\varphi_{k,r}(\omega)\right\rangle \right|^{2},
\]
with the double series converging monotonically.
\item The family 
\[
\left\{ \varphi_{k,r}(\omega):k\ge1,\ r\in\mathcal{R}_{k}(\omega)\right\} 
\]
is a frame for $\overline{ran}(R^{1/2}_{0})$ with frame operator
$R_{0}$: for every $x\in H$, 
\[
\sum_{k,r}\left|\left\langle x,\varphi_{k,r}(\omega)\right\rangle \right|^{2}=\left\langle x,R_{0}x\right\rangle ,
\]
and the closed linear span of $\{\varphi_{k,r}(\omega)\}$ is $\overline{ran}(R^{1/2}_{0})$.
In particular, when $R_{0}=I$ on a closed subspace $K\subset H$,
this is a Parseval frame for $K$ in the usual sense.
\end{enumerate}
\end{thm}

\begin{proof}
By \prettyref{thm:7-1}, there exists a $\nu^{\mathrm{tr}}$-full
measure set $\Omega_{0}\subset\Omega$ such that $R_{\infty}(\omega)=0$
for every $\omega\in\Omega_{0}$.

For each $\omega\in\Omega$ and $n\ge1$, the pathwise telescoping
identity gives 
\[
R_{\omega|n}=R_{0}-\sum^{n}_{k=1}\Delta_{k}(\omega),
\]
with convergence in the strong operator topology as $n\to\infty$,
and $R_{\omega|n}\downarrow R_{\infty}(\omega)$ in the Loewner order.
Thus for each $\omega$, 
\[
R_{0}-R_{\infty}(\omega)=\sum^{\infty}_{k=1}\Delta_{k}(\omega)
\]
in the strong operator topology, with the series converging monotonically
in the Loewner order.

If $\omega\in\Omega_{0}$, then $R_{\infty}(\omega)=0$, so 
\[
R_{0}=\sum^{\infty}_{k=1}\Delta_{k}(\omega)
\]
in the strong operator topology. This proves part (1).

Fix $\omega\in\Omega_{0}$ and $x\in H$. Applying $\left\langle x,\cdot x\right\rangle $
to the identity in (1) yields 
\[
\left\langle x,R_{0}x\right\rangle =\sum^{\infty}_{k=1}\left\langle x,\Delta_{k}(\omega)x\right\rangle ,
\]
with monotone convergence in $k$, since the partial sums are increasing
and bounded above by $\left\langle x,R_{0}x\right\rangle $.

For each $k$, 
\[
\left\langle x,\Delta_{k}(\omega)x\right\rangle =\sum_{r\in\mathcal{R}_{k}(\omega)}\left|\left\langle x,\varphi_{k,r}(\omega)\right\rangle \right|^{2},
\]
with monotone convergence in $r$. Combining these relations gives
\[
\left\langle x,R_{0}x\right\rangle =\sum^{\infty}_{k=1}\sum_{r\in\mathcal{R}_{k}(\omega)}\left|\left\langle x,\varphi_{k,r}(\omega)\right\rangle \right|^{2}.
\]
Since all terms in the double series are nonnegative, the convergence
is monotone and the order of summation may be rearranged arbitrarily:
\[
\sum_{k,r}\left|\left\langle x,\varphi_{k,r}(\omega)\right\rangle \right|^{2}=\left\langle x,R_{0}x\right\rangle .
\]
This proves part (2).

For part (3), first note that each $\varphi_{k,r}(\omega)$ lies in
$ran(R^{1/2}_{\omega|k-1})\subseteq ran(R^{1/2}_{0})$, so 
\[
V(\omega):=\overline{span}\left\{ \varphi_{k,r}(\omega):k\ge1,\ r\in\mathcal{R}_{k}(\omega)\right\} \subseteq\overline{ran}(R^{1/2}_{0}).
\]

Conversely, suppose $y\in H$ satisfies $y\perp V(\omega)$. Then
$\left\langle y,\varphi_{k,r}(\omega)\right\rangle =0$ for all $(k,r)$,
and by part (2) we obtain 
\[
\sum_{k,r}\left|\left\langle y,\varphi_{k,r}(\omega)\right\rangle \right|^{2}=\left\langle y,R_{0}y\right\rangle .
\]
The left-hand side vanishes, so $\left\langle y,R_{0}y\right\rangle =0$.
Since $R_{0}\ge0$, this implies $R^{1/2}_{0}y=0$, that is, $y\in\ker R_{0}$.

Thus we have shown that 
\[
V(\omega)^{\perp}\subseteq\ker R_{0}.
\]
Taking orthogonal complements yields 
\[
(\ker R_{0})^{\perp}\subseteq V(\omega).
\]
But $(\ker R_{0})^{\perp}=\overline{ran}(R^{1/2}_{0})$. Together
with the inclusion $V(\omega)\subseteq\overline{ran}(R^{1/2}_{0})$
obtained above, this implies 
\[
V(\omega)=\overline{ran}(R^{1/2}_{0}).
\]

Combining this with the identity in part (2), we see that $\{\varphi_{k,r}(\omega)\}$
is a frame for $\overline{ran}(R^{1/2}_{0})$ with frame operator
$R_{0}$: for every $x\in H$, 
\[
\sum_{k,r}\left|\left\langle x,\varphi_{k,r}(\omega)\right\rangle \right|^{2}=\left\langle x,R_{0}x\right\rangle .
\]
In particular, when $R_{0}$ is the identity operator on a closed
subspace $K\subset H$ (and zero on $K^{\perp}$), this identify reduces
to 
\[
\sum_{k,r}\left|\left\langle x,\varphi_{k,r}(\omega)\right\rangle \right|^{2}=\left\langle x,x\right\rangle ,\quad x\in K,
\]
so $\{\varphi_{k,r}(\omega)\}$ is a Parseval frame for $K$.
\end{proof}
\bibliographystyle{amsalpha}
\bibliography{ref}

\end{document}